\documentclass[a4paper,10pt,reqno, english]{amsart}  
\usepackage[utf8]{inputenc}
\usepackage{amssymb}
\usepackage{graphicx}
\usepackage{amsmath,amsthm}
\usepackage{amsfonts,amssymb,enumerate}
\usepackage{url,paralist}
\usepackage{mathtools}  
\usepackage{tikz}
\usepackage{pdfsync}
\usepackage{hyperref}
\usepackage{color}
\usepackage{enumerate,amssymb}
\usepackage{caption}
\usepackage{subcaption}
\usepackage{mathtools} 
\usepackage{geometry}
\usepackage[arrow,curve,matrix,tips,2cell]{xy}  
\SelectTips{eu}{10} \UseTips
\UseAllTwocells
\usepackage{setspace}
\usepackage{cleveref}

\theoremstyle{plain}
\newtheorem{theorem}{Theorem}[section]

\newtheorem{lemma}[theorem]{Lemma}

\newtheorem{corollary}[theorem]{Corollary}
\newtheorem{proposition}[theorem]{Proposition}

\newtheorem{conjecture}[theorem]{Conjecture}
\newtheorem*{conjecture*}{Conjecture}
\newtheorem*{theorem*}{Theorem}

\newtheorem*{claim*}{Claim}

\theoremstyle{definition}
\newtheorem{definition}[theorem]{Definition}

\newtheorem{remark}[theorem]{Remark}

\hypersetup{
  colorlinks = false,
  urlcolor = black,
  pdfkeywords = {albert haase, mathematics, mathematik},
  pdfsubject = {},
  pdfpagemode = UseNone
}

\newcommand{\A}{\mathcal{A}}

\newcommand{\R}{\mathbb{R}}
 
\newcommand{\N}{\mathbb{N}}
\newcommand{\Z}{\mathbb{Z}}
\newcommand{\F}{\mathbb{F}}

\DeclareMathOperator{\TT}{TT}

\begin{document}

\title[Tverberg-type theorems for matroids]
{Tverberg-type theorems for matroids:\\A counterexample and a proof}

\author[Blagojevic]{Pavle V. M. Blagojevi\'{c}}
\thanks{P.V.M.B.\ received funding from DFG via the Collaborative Research Center TRR~109 ``Discretization in Geometry and Dynamics'' and the grant ON 174008 of the Serbian Ministry of Education and Science.} 
\address{Inst. Math., FU Berlin, Arnimallee 2, 14195 Berlin, Germany} 
\email{blagojevic@math.fu-berlin.de}
\author[Haase]{Albert Haase} 
\thanks{A.H. was supported by DFG via the Berlin Mathematical School.}
\address{Inst. Math., FU Berlin, Arnimallee 2, 14195 Berlin, Germany} 
\email{a.haase@fu-berlin.de}
\author[Ziegler]{G\"unter M. Ziegler} 
\thanks{G.M.Z.\ received funding from DFG via the Collaborative Research Center TRR~109 ``Discretization in Geometry and Dynamics''}
\address{Inst. Math., FU Berlin, Arnimallee 2, 14195 Berlin, Germany} 
\email{ziegler@math.fu-berlin.de}

\date{10. May 2017}

\begin{abstract}
B\'ar\'any, Kalai, and Meshulam recently obtained a topological Tverberg-type theorem for matroids, which guarantees multiple coincidences for continuous maps from a matroid complex to~$\R^d$, if the matroid has sufficiently many disjoint bases. They make a conjecture on the connectivity of $k$-fold deleted joins of a matroid with many disjoint bases, which would yield a much tighter result -- but we provide a counterexample already for the case of $k=2$, where a tight Tverberg-type theorem would be a topological Radon theorem for matroids. Nevertheless, we prove the topological Radon theorem for the counterexample family of matroids by an index calculation, despite the failure of the connectivity-based approach.
\end{abstract}

\maketitle

\section{Introduction and statement of main results}
\label{sec:introduction}

Let $d\geq 1$ and $k\geq 1$ be integers and let $f\colon \Sigma \rightarrow \R^d$ be a continuous map from a non-trivial simplicial complex~$\Sigma$ to~$\R^d$. A \emph{Tverberg $k$-partition of~$f$} is a collection $\{\sigma_1, \dots, \sigma_k \}$ of $k$ pairwise disjoint faces of~$\Sigma$ such that $\bigcap_{i=1}^k f(\sigma_i) \neq \emptyset$. For fixed $d\geq 1$, the \emph{topological Tverberg number} $\TT(\Sigma,d)$ is the maximal integer~$k\geq 1$ such that every continuous map $f \colon \Sigma \rightarrow \R^d$ has a Tverberg $k$-partition. The topological Tverberg theorem due to B{\'{a}}r{\'{a}}ny, Shlosman, and Sz{\H{u}}cs~\cite{barany_shlosman_topol_tverberg1981} implies that, if~$\Sigma$ is the $d$-skeleton~$\Delta_{(k-1)(d+1)}^{(d)}$ of the simplex of dimension~$(k-1)(d+1)$ and~$k$ is prime, then $\TT(\Delta_{(k-1)(d+1)}^{(d)},d) = k$. For $k=2$ this result is equivalent to the topological Radon theorem~\cite{bajmoczy_barany_topo_radon1979}. It follows from the work of {\"{O}}zaydin \cite{ozaydin_1987} that this result remains true when~$k$ is a prime power. Recently Frick \cite{frick_counterexamples2015} \cite{frick_barycenters_and_counterexamples2015}, using the ``constraint method'' \cite{blagojevic_frick_ziegler_tverberg_plus_constraints2014} and building on work  by Mabillard and Wagner \cite{mabillard_wagner_elim_I_2015}, showed that if~$k\geq 6$ is not a prime-power and~$d\geq 3k+1$, then $\TT(\Delta_{(k-1)(d+1)},d) < k$; see~\cite{barany_blagojevic_ziegler_tverbergs_thm_at_50_2016} for a recent survey.

Recently by B{\'{a}}r{\'{a}}ny, Kalai, and Meshulam \cite{barany_kalai_meshulam2016} gave lower bounds for the topological Tverberg number of a matroid, regarded as the simplicial complex of its independent sets. Let~$\Sigma$ be a matroid~$M$ of rank~$d+1$ with~$b$ disjoint bases, then \cite[Thm.\,1]{barany_kalai_meshulam2016} asserts that $\TT(M,d) \geq \sqrt{b}/4$. If $M$ is the uniform matroid~$\Delta_{(k-1)(d+1)}^{(d)}$, then this result implies that $\TT(\Delta_{(k-1)(d+1)}^{(d)},d) \geq \sqrt{k-1}/4$ for all integers $d,k \geq1$.

The results of \cite{barany_shlosman_topol_tverberg1981}, \cite{ozaydin_1987}, and \cite[Thm.\,1]{barany_kalai_meshulam2016} mentioned above are all obtained by using a ``configuration space/test map scheme.'' This approach involves defining for each~$k\geq 1$ a space~$X$ related to the complex~$\Sigma$, called the \emph{configuration space}, and a space~$Y$ related to~$\R^d$, called the \emph{test space}, such that both spaces admit an action by a finite non-trivial group~$G$. A continuous map \emph{without} a Tverberg $k$-partition then defines a related $G$-equivariant map $X \rightarrow Y$, called the \emph{test map}. The method of proof is to show the non-existence of such a $G$-equivariant map~$X \rightarrow Y$. The configuration space test map scheme is a classical and frequently used tool to solve problems in combinatorics and discrete geometry; see Matou\v{s}ek \cite{matousek_using_BU_2nd_2003} for an introduction and a survey of applications.

There are two major types of configuration test map scheme that
have proved to be particularly powerful, based on
using joins resp.\ products as configuration spaces.
In the \emph{join scheme} used in \cite{sarkaria_a_generalized_vk_flores1991} and \cite{barany_kalai_meshulam2016} the configuration space~$X$ is the $k$-fold deleted join~$\Sigma^{*k}_\Delta$ of the complex~$\Sigma$ and the test space~$Y$ is a sphere~$S^{(k-1)(d+1) -1}$ of dimension~$(k-1)(d+1) -1$. In the \emph{product scheme} used in \cite{barany_shlosman_topol_tverberg1981} the configuration space~$X$ is the $k$-fold deleted product~$\Sigma^{\times k}_\Delta$ of the complex~$\Sigma$ and the test space~$Y$ is a sphere~$S^{(k-1)d}$. If~$k$ is prime, then all spaces and in particular the two spheres admit free actions by the group~$\Z/k$. 

In order to obtain sharp results using a configuration space/test map scheme it is necessary to determine proof strategies for the non-existence of an equivariant map from the configuration space~$X$ to the test space~$Y$. One commonly used method is the \emph{connectivity-based approach}, which can be applied if~$Y$ is a finite-dimensional CW complex on which the group acts freely: If one establishes that the connectivity of the space~$X$ is at least as high as the dimension of the space~$Y$, then Dold's theorem \cite{dold_simple_proofs_borsuk1983} implies that an equivariant map~$X\rightarrow Y$ does not exist. For a more general version of Dold's theorem that is also applicable in this context see~\cite{volovikov_on_a_top_general1996}.

The connectivity-based approach (for $k$ a prime power) yields tight bounds for the topological Tverberg number of~$\Sigma =\Delta_{(k-1)(d+1)}^{(d)}$ with both the product scheme \cite{barany_shlosman_topol_tverberg1981} and the join scheme \cite{sarkaria_a_generalized_vk_flores1991}. 
It is natural to consider the more general situation when~$\Sigma$ is a well-behaved simplicial complex, say a matroid~$M$: What is the connectivity of the configuration spaces? Which results can/\allowbreak cannot be obtained via a connectivity-based approach? Having these questions in mind, B{\'{a}}r{\'{a}}ny, Kalai, and Meshulam formulated the following conjecture.

\begin{conjecture}[{{B\'{a}r\'{a}ny}, Kalai, and Meshulam 2016 \cite[Conj.\,4]{barany_kalai_meshulam2016}}]\label{conj:bkm_barany_kalai}
For any integer~$k\geq1$ there exists an integer~$n_k \geq 1$ depending only on~$k$ such that for any matroid~$M$ of rank~$r\geq 1$ with at least~$n_k$ disjoint bases, the $k$-fold deleted join~$M^{*k}_\Delta$ of the matroid~$M$ is $(kr-1)$-dimensional and $(kr-2)$-connected.
\end{conjecture}

For~$k=1$ the conjecture is true, since a matroid of rank~$r$ is pure shellable and hence in particular $(r-2)$-connected \cite[Thm.\,4.1]{bjorner_wachs_I_nonpure_shellable1996}.  Using the connectivity-based approach the conjecture would imply that for a matroid~$M$ of rank~$d+1$ with~$b \geq n_k$ disjoint bases the topological Tverberg number satisfies~$\TT(M,d)\geq k$.

We prove the following theorem that gives a counterexample to the conjecture already in the case where~$k=2$. 

\begin{theorem}[\Cref{conj:bkm_barany_kalai} fails for $k=2$]\label{thm:bkm_main_counterexample}
There is a family of matroids $M_r\,(r \in \Z, r \geq 2)$ such that each matroid~$M_r$ has rank~$r$ and $r$~disjoint bases, while the $2$-fold deleted join~$(M_r)^{*2}_\Delta$ of~$M_r$ is is $(2r-1)$-dimensional and $(2r-3)$-connected, but not~$(2r-2)$-connected. 
\end{theorem}

The family of matroids $M_r$ ($r\geq 2)$ is a tight example for the failure of \Cref{conj:bkm_barany_kalai} in the sense that if we increase the number of bases from $r$ to $r+1$, then the $2$-fold deleted join of the new complex is $(2r-2)$-connected; see \Cref{cor:bkm_increase_num_bases_M_r}. 
To prove \Cref{thm:bkm_main_counterexample} we first show that the complex~$(M_r)^{*2}_\Delta$ is shellable for $r\geq 3$ using the notion of shellability for non-pure complexes due to Björner and Wachs \cite{bjorner_wachs_I_nonpure_shellable1996} \cite{bjorner_wachs_II_nonpure_shellable1997}; see \Cref{prop:bkm_shellability_M_r_2}. The crucial ingredient in the proof is  \Cref{prop:bkm_balanced_subcomplexes_shellable}, which shows that balanced subcomplexes of shellable balanced complexes are again shellable. The case $r=2$ is treated separately; see \Cref{rem:bkm_case_M_r_r=2}. We give a first proof of \Cref{thm:bkm_main_counterexample} by constructing a covering of~$(M_r)^{*2}_\Delta$ by two subcomplexes; see~\Cref{cor:bkm_htpy_type_M_r_2}. A second proof of \Cref{thm:bkm_main_counterexample} is a straightforward calculation involving only the combinatorics of~$(M_r)^{*2}_\Delta$; see \Cref{subsec:bkm_proof_main_thm_1}. This allows us to calculate the Betti numbers of~$(M_r)^{*2}_\Delta$; see \Cref{cor:bkm_betti_nums_M_r_2}.

Using the connectivity-based approach one obtains that $\TT(M_r,d) \geq 2$ when $2r-3 \geq d$; see \Cref{cor:bkm_radon_for_M_r_via_connectivity}.
However, despite the lower connectivity of the matroid~$M_r$ we still obtain a sharp topological Radon theorem for~$M_r$ by means of a Fadell--Husseini index argument that goes back to \cite[Thm.\,1]{blagojevic_blagojevic_mcclary_equi_tian_jordan_curve} and \cite[Thm.\,4.2]{blagojevic_lueck_ziegler_equiv_top_config_spaces2015}; for the classical reference regarding the Fadell--Husseini index see \cite{husseini_ideal_val_index1988}. Thus the following theorem is an example of a Tverberg-type result for a family of matroids that cannot be obtained via the connectivity-based approach.

\begin{theorem}[Topological Radon theorem for~$M_r$]\label{thm:bkm_radon_for_M_r}
Let $d\geq 1$ and $r\geq 3$ be integers such that $2r -2\geq  d$. Then the topological Tverberg number of the family of matroids~$M_r$ from \Cref{thm:bkm_main_counterexample} satisfies $\TT(M_r,d) \geq 2$.
\end{theorem}

We summarize the remaining results of this chapter as follows.
\begin{compactitem}
\item We show that \cite[Cor.\,3]{barany_kalai_meshulam2016} in fact implies lower bounds for the topological Tverberg number~$\TT(M,d)$ for matroids~$M$ of all ranks; see \Cref{cor:lower_bounds_TT_num}.
\item We give upper bounds for the topological Tverberg number~$\TT(M,d)$ in the case where the rank~$r$ of the matroid~$M$ is at most $d-2$; see \Cref{prop:bkm_upper_bounds_TT_num}.
\item We show that the connectivity of the $k$-fold deleted product~$M^{\times k}_\Delta$ of a matroid~$M$ of rank~$r$ with $b$~disjoint bases is at least $r-2 - \lfloor r(k-1)/b \rfloor$, when $k\geq 2$ and $b,r\geq k$. If $b \geq r(k-1)+1$, then~$M^{\times k}_\Delta$ is not $(r-1)$-connected; see \Cref{thm:bkm_conn_del_prod}.
\item Using \Cref{thm:bkm_conn_del_prod} we establish the connectivity of the ordered configuration space of two particles in a matroid; see \Cref{cor:bkm_conn_config_space}. 
\end{compactitem}

\section{Preliminaries}
\label{sec:bkm_prelims}

\subsection{Terminology}
\label{subsec:bkm_terminology}

By a \emph{simplicial complex} or simply \emph{complex} we refer to a finite abstract simplicial complex or a geometric realization of a finite abstract simplicial complex. We require that any complex contains the empty set as a face of dimension~$-1$. A \emph{facet} of a complex is a face that is not contained in any other  face.  Let $\Sigma_1, \dots, \Sigma_k$  be simplicial complexes with  vertex sets $V_1, \dots, V_k$. Then the \emph{join} of  the~$\Sigma_i$ is defined as the simplicial complex $\Sigma_1 * \dots * \Sigma_k = \{ \sigma_1 \sqcup \dots \sqcup \sigma_k \, \colon \,\sigma_i \in \Sigma_i\}$ with  vertex set equal to the disjoint union~$\bigsqcup_{i=1}^k V_i$. Assume the vertex sets~$V_i$ are all contained in a common set~$V$, then the  \emph{deleted join} of the~$\Sigma_i$ is defined as the simplicial complex $(\Sigma_1 * \dots *\Sigma_k)_\Delta = \{\sigma_1\sqcup  \dots \sqcup \sigma_k \, :\, \sigma_i \in \Sigma_i, \, \sigma_i \cap \sigma_j = \emptyset \text{ for } i\neq j\}$ with vertex set  $\bigsqcup_{i=1}^k V_i$. Let  $\Sigma_i=\Sigma$ for $i=1,\dots,k$. Then $\Sigma^{*k}:=\Sigma_1 * \dots *\Sigma_k$ is the \emph{$k$-fold join} of~$\Sigma$ and $\Sigma_\Delta^{*k}:=(\Sigma_1* \dots * \Sigma_k)_\Delta$ is the \emph{$k$-fold deleted join} of~$\Sigma$. If $\sigma \subseteq V$, the \emph{deletion} of $\sigma$ from $\Sigma$ is defined as $\Sigma\setminus \sigma = \{ \tau \in \Sigma \, \colon \, \sigma \not\subseteq \tau\}$. We also denote $\Sigma\setminus \sigma$ by $\Sigma | (V\setminus \sigma)$ and refer to it as the \emph{restriction} of $\Sigma$ to the set $V\setminus \sigma$. The \emph{link} of $\Sigma$ with respect to a face $\sigma\in \Sigma$ is defined as $\Sigma/\sigma=\{ \tau \in \Sigma \, \colon \, \sigma \cap \tau = \emptyset, \sigma \cup \tau \in \Sigma\}$. Given a geometric simplicial complex~$\Sigma$, we define the \emph{$k$-fold deleted product} $\Sigma^{\times k}_\Delta$ of~$\Sigma$ as the CW complex  with cells given by  products of relative interiors of (geometric) simplices~$\sigma_i \in \Sigma$  of the form $\mathrm{relint}(\sigma_1)\times \dots \times \mathrm{relint}(\sigma_k)$, where $\sigma_i \cap \sigma_j = \emptyset$ for all $i,j$ with $1\le i<j\le k$. The attaching maps for $\Sigma^{\times k}_\Delta$ are given by the products of the attaching maps of~$\Sigma$. For additional terminology and results regarding simplicial complexes see Matou\v{s}ek \cite{matousek_using_BU_2nd_2003}. 
  
A \emph{matroid}~$M$ with \emph{ground set}~$E$ is a simplicial complex with vertices in~$E$ such that for every $A\subseteq E$ the restriction $M|A = \{\sigma \in M \colon \sigma \subseteq A\}$ is pure. We  call a face of~$M$ an~\emph{independent set}, a facet of~$M$ a \emph{basis}, and  the cardinality of a (any) basis the \emph{rank} of~$M$.
Let~$m$ and~$n$ be integers with $0 \le m\le n$. Given a ground set~$E$ of cardinality~$n$, the \emph{uniform matroid $U_{m,n}(E)$} is given by the collection of all subsets of~$E$ of cardinality at most~$m$. Let~$\Delta_{n-1}^{(m-1)}$ be $(m-1)$-skeleton of the simplex of dimension~$n-1$. Then we have $\Delta_{n-1}^{(m-1)}=U_{m,n}(E)$. Given matroids~$M_1, \dots, M_k$ with ground sets~$E_1, \dots, E_k$, the \emph{direct sum} $M_1 \oplus  \dots \oplus M_k$ of the family~$M_i$ is defined as the collection $\{ I_1 \sqcup  \dots \sqcup I_k \,\colon \, I_i \in M_i \}$ and is a matroid with ground set~$E_1\sqcup \dots \sqcup E_k$. The direct sum of a collection of matroids is equal to the join of the collection of matroids, viewed as simplicial complexes. For additional terminology and results regarding matroids see Oxley \cite{oxley_matroid_theory_2011}.

\subsection{Non-pure shellability}
\label{subsec:bkm_non_pure_shellability}
Since some of the complexes we are interested in are non-pure, we use the notions of ``non-pure shellability'' introduced by Bj\"orner and Wachs \cite{bjorner_wachs_I_nonpure_shellable1996} \cite{bjorner_wachs_II_nonpure_shellable1997}.

By \cite[Def.\,2.1]{bjorner_wachs_I_nonpure_shellable1996} a \emph{shelling} of a possibly non-pure finite simplicial complex~$\Sigma$ of dimension~$d$ is defined as a strict order~``$\ll$'' on the set $\mathcal{F}$ of facets of~$\Sigma$ such that for any facet $B\in \mathcal{F}$~of dimension~$d'\le d$ for which there exists a prior facet $A \in \mathcal{F}$ with $A \ll B$, the simplicial complex 
\[
B\cap \big(\bigcup_{A\in \mathcal{F},\, A \ll B} A\; \big) 
\] 
defined by the intersection of~$B$ with the union of the previous facets (and their faces) is pure and $(d'-1)$-dimensional. This is equivalent to the following condition. For any two facets $A, B \in \mathcal{F}$ with $A\ll B$, there is a facet $C \in \mathcal{F}$ and a vertex $x \in B$ such that 
\begin{equation}
C \ll B \quad \text{and} \quad A \cap B \subseteq B\cap C = B \setminus \{v\}. \label{eq:bkm_shelling_condition}
\end{equation}
For pure complexes~$\Sigma$, the above definition coincides with the ``usual'' definition of shellability. A $d$-dimensional simplicial complex~$\Sigma$ is \emph{shellable} if it has a shelling. It is \emph{pure shellable} if it is pure and shellable. 

\section{Proof of the main result}
\label{sec:bkm_proof_main_result} 
\subsection{The counterexample family $M_r$}
\label{subsec:bkm_counterexample_intro}
\begin{definition}[The counterexample family~$M_r$] \label{def:bkm_matroid_Mr}
Let $r \geq 2$ be an integer. Let~$E$ be a set of pairwise distinct elements $v_i^j$ and~$w_j$ for $i=1,\dots,r-1$ and~$j=1,\dots,r$. Define \emph{blocks}~$E_i$ by
\[
E_i = \{v_i^1, \dots, v_i^{r}\} \quad \text{for} \quad i =1, \dots, r-1, \quad \text{and} \quad E_r = \{w_1, \dots, w_r\}.
\]
Define a matroid~$\widehat{M_r}$ by
\[
\widehat{M_r} = U_{1,r} (E_1) \oplus  \dots \oplus U_{1,r} (E_{r-1}) \oplus U_{r,r} (E_r).
\]
Then the matroid~$M_r$ with ground set~$E$ is defined as the $(r-1)$-skeleton of~$\widehat{M_r}$, 
hence \\ $M_r = \{I \in \widehat{M_r}\, \colon \, |I| \le r\}.$
\end{definition}

The matroid~$M_r$ has rank $r$ and has~$r$ pairwise disjoint bases of the form $\{v_1^j,\dots, v_{r-1}^j,w_j\}$ for $j=1,\dots,r$.  Faces of $M_r$ are given by choosing at most~$r$ vertices in total and at most~$1$~vertex in each of the first $r-1$~blocks; see \Cref{fig:bkm_matroid_mr}.

\begin{figure}[h!]
\begin{center}
\includegraphics[scale=.87]{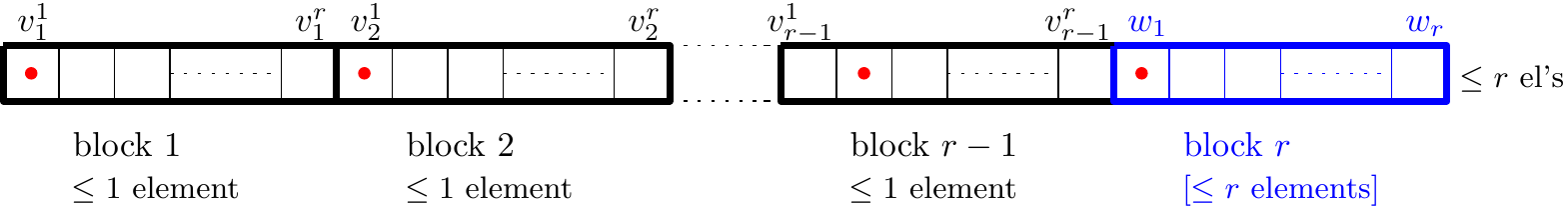}
\caption{The matroid~$M_r$.}
\label{fig:bkm_matroid_mr}
\end{center}
\end{figure} 

\begin{figure}[h!]
\begin{center}
\includegraphics[scale=.87]{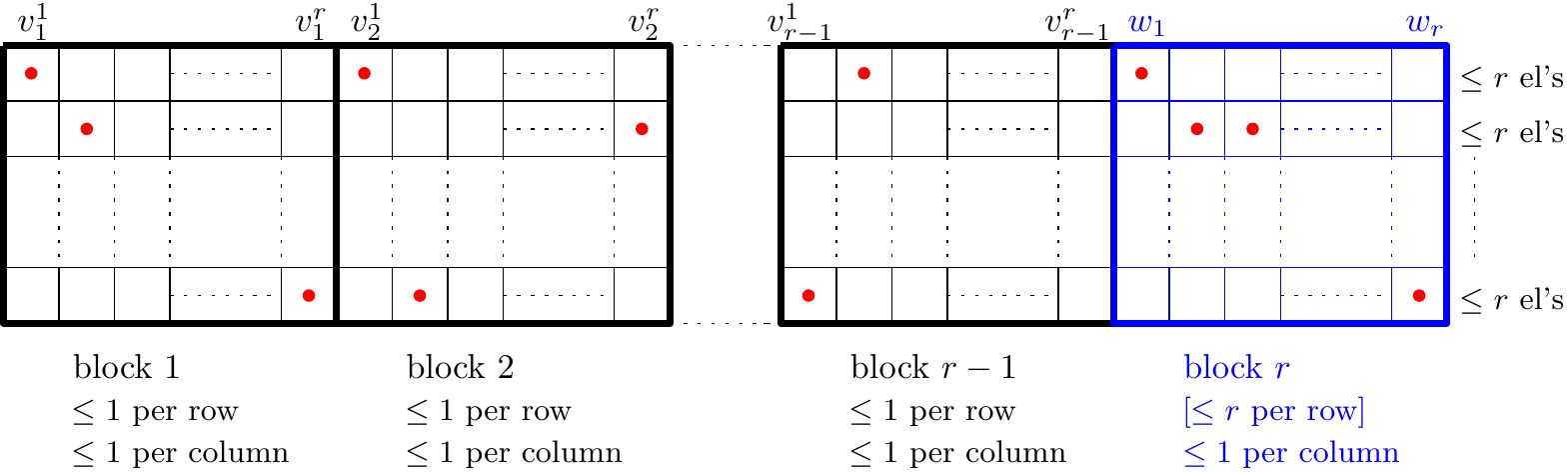}
\caption{The $k$-fold deleted join $(M_r)^{*k}_\Delta$ with an example facet.}
\label{fig:bkm_matroid_deljoin_mr}
\end{center}
\end{figure}
 
Consider the $k$-wise deleted join of the complex~$M_r$, which we denote by~$(M_r)^{*k}_\Delta$. We display the vertices of~$(M_r)^{*k}_\Delta$ in~$k$ \emph{rows}, based on the copy of~$M_r$ they belong to. We group the vertices of~$(M_r)^{*k}_\Delta$ into~$r$ \emph{blocks}; see \Cref{fig:bkm_matroid_deljoin_mr}. A \emph{column} of~$(M_r)^{*k}_\Delta$ consists of the $k$~copies of a fixed vertex $v\in E$. Faces of~$(M_r)^{*k}_\Delta$ are given by choosing at most~$r$~vertices in each row, at most~$1$ vertex per column and at most~$1$~vertex in each row of each of the first $r-1$~blocks.  Note that~$(M_r)^{*k}_\Delta$ has dimension $d=2r-1$ and is not pure: Its facets have dimensions $d,d-1,\dots, d-k+1$. See \Cref{fig:bkm_matroid_deljoin_m5_8_facet} for an example facet of dimension $d-1=8$ for~$r=5$ and~$k=2$.

\begin{figure}[h!]
\begin{center}
\includegraphics[scale=.91]{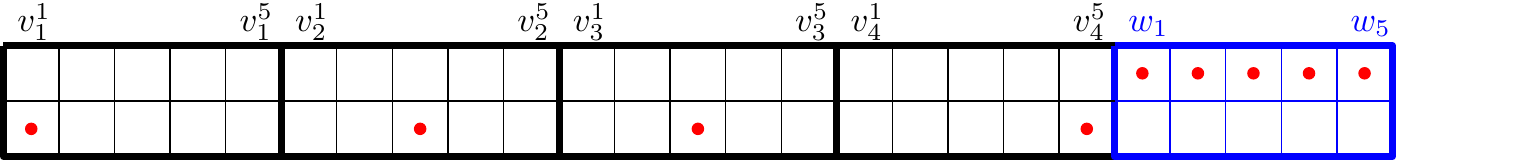}
\caption{An $8$-dimensional facet of the $9$-dimensional complex~$(M_5)^{*2}_\Delta$.}
\label{fig:bkm_matroid_deljoin_m5_8_facet}
\end{center}
\end{figure}

\subsection{Shellability of subcomplexes of balanced complexes}
\label{subsec:bkm_shellability_subcomplexes_balanced}

To define a shelling of~$(M_r)^{*2}_\Delta$ we use the existence of pure shellings of certain pure subcomplexes that we can describe as ``balanced complexes.'' Let us recall the definition of a balanced complex.

\begin{definition}[Stanley {\cite[Sec.\,2]{stanley_balanced_cohen_macaulay1979}}]\label{def:bkm_balanced_complex}
Let~$m\geq 1$  and~$d\geq 0$ be integers and let $a=(a_1,\dots, a_m)$ be an $m$-tuple of non-negative integers such that $a_1 +\dots + a_m = d+1$. Let~$\Sigma$ be a $d$-dimensional simplicial complex with vertex set~$V$. Let $\mathcal{V}:=(V_1, \dots, V_m)$ be an ordered partition of~$V$ into pairwise disjoint sets~$V_i$, called a  \emph{vertex coloring}. We call~$\Sigma$ a \emph{balanced complex (of type $a$ with respect to the partition~$\mathcal{V}$)} if
\begin{compactenum}[(i)]
\item $\Sigma$ is pure, and
\item for every facet $A\in \Sigma$ we have that $|A\cap V_i |= a_i$ for $i=1,\dots, m$.
\end{compactenum}
We call $\Sigma$ \emph{completely balanced} if it is balanced of type $(1, \dots, 1)$. 
\end{definition}

For example, the order complex of any graded poset is completely balanced. A simplicial complex is pure if and only if it is balanced of type $a=(a_1)$. Balanced complexes  were introduced by Stanley in 1979 \cite{stanley_balanced_cohen_macaulay1979}. They have been studied in the context of posets \cite{baclawski_CM_ordered_sets1980} \cite{bjorner_wachs_on_lex_shell_posets1983}, simplicial polytopes  \cite{goff_Klee_Novik_Balanced_complexes_without_large2011} \cite{juhnke-kubitzke_murai_balanced_generalized_lower_bound2015}, and Cohen-Macaulay or shellable complexes \cite{bjorner_wachs_I_nonpure_shellable1996} \cite{mnukhin_saturated_balanced2005}. We point out that some authors use ``balanced complex'' to refer to a ``completely balanced complex.'' 

Balanced complexes are not necessarily pure shellable, as can be seen by taking any pure non-shellable $d$-dimensional complex. Given a balanced complex~$\Sigma$ consider its \emph{type-selected subcomplex}~$\Sigma_T$ \cite[p.\,1858]{bjorner_topological_methods_chapter1995}, which is the restriction of~$\Sigma$ to the set~$\bigcup_{i \in T} V_i$ of vertices whose types (colors) are contained in the set~$T \subseteq \{1,\dots, m\}$. It was shown in \cite{bjorner_shellable_CM_posets1980} that any type-selected subcomplex of a pure shellable complex is pure shellable; see \cite[Thm.\,11.13]{bjorner_topological_methods_chapter1995} for a more general result. However, we are interested in the pure shellability of the following subcomplex.

\begin{definition}[Balanced $b$-skeleton]\label{def:bkm_balanced_k_skeleton}
Let $m\geq 1$ be an integer and let~$\Sigma$ be a balanced $d$-complex of type $a=(a_1, \dots, a_m)$ with vertex coloring $(V_1,\dots, V_m)$ and let $b = (b_1, \dots ,b_m)$ be an $m$-tuple of integers with $0 \le b_i \le a_i$ for $i=1, \dots, m$. Then the complex~$\Sigma^b$ given by the faces~$F$ of~$\Sigma$ for which $|F\cap V_i| \le b_i$ for $i=1,\dots, m$ is the \emph{balanced $b$-skeleton of~$\Sigma$}.
\end{definition}

To show that balanced $b$-skeleta of pure shellable balanced complexes are pure shellable (\Cref{prop:bkm_balanced_subcomplexes_shellable}) we use the existence of shellings of the skeleta of a shellable complex that are ``compatible'' with the  original shelling of the complex. Let us formulate this as a definition.

\begin{definition}\label{def:bkm_compatible_shelling}
Let~$\Sigma$ be a shellable simplicial complex of dimension~$d\geq 0$ with shelling order~``$\ll$.'' Let~$k$ be an integer with~$0 \le k \le d$. Then a shelling~$\ll'$ of the $k$-skeleton of~$\Sigma$ is  \emph{compatible with the shelling of~$\Sigma$} if the following implication holds for any two $k$-faces~$\overline{A}$ and $\overline{B}$ of~$\Sigma$: If $\overline{A} \ll'\overline{B}$ and if~$A$ and~$B$ are the smallest (w.r.t.~``$\ll$'') $d$-faces of~$\Sigma$ containing~$\overline{A}$ and $\overline{B}$, respectively, then $A\ll B$ or $A=B$.
\end{definition}

The following lemma is true even for shellable complexes that are non-pure. We apply it only in the pure setting.

\begin{lemma}[Bj\"orner and Wachs {\cite[Thm.\,2.9]{bjorner_wachs_I_nonpure_shellable1996}}]\label{lem:bkm_compatible_shellings_exist}
Let~$\Sigma$ be a shellable simplicial complex of dimension~$d$. Let~$k$ be an integer with $0\le k\le d$. Then there exists a shelling of the $k$-skeleton of~$\Sigma$ that is compatible with the shelling of~$\Sigma$.
\end{lemma}

We point out that the property of compatibility is transitive in the following sense. Let~$\Sigma$ be a shellable complex and let the shelling of its $k$-skeleton be compatible with the shelling of~$\Sigma$. Then any shelling of its $(k-1)$-skeleton that is compatible with the shelling of its $k$-skeleton is also compatible with the shelling of~$\Sigma$. 

\begin{proposition}\label{prop:bkm_balanced_subcomplexes_shellable}
Let~$m\geq 1$ be an integer and let~$\Sigma$ be a balanced $d$-complex of type $a=(a_1, \dots, a_m)$ with vertex coloring $(V_1,\dots, V_m)$ and let $b = (b_1, \dots ,b_m)$ be an $m$-tuple of non-negative integers with $0 \le b_i \le a_i$ for $i=1, \dots, m$. If~$\Sigma$ is pure shellable, then the balanced $b$-skeleton~$\Sigma^b$ of~$\Sigma$ is  pure shellable.
\end{proposition}
\begin{proof}
By induction and transitivity of compatibility, it suffices to prove the statement for $b_1= a_1-1$ and $b_i = a_i$ for $i=2\dots, m$. Let $d$ denote the dimension of $\Sigma$. Let~``$\ll$'' denote the shelling of~$\Sigma$ and, to simplify notation, let~``$\ll$'' also denote a compatible shelling of the $(d-1)$-skeleton of~$\Sigma$, which exists by \Cref{lem:bkm_compatible_shellings_exist}. We  show that the restriction of this shelling to the subcomplex~$\Sigma^b$ is a shelling.  Let $\overline{A}$ and $\overline{B}$ be two (balanced) facets of~$\Sigma^b$ such that $\overline{A} \ll \overline{B}$. We must show that there exists a $(d-1)$-face $\overline{C}$ of $\Sigma$ and a vertex $v \in \overline{B}$ such that
\begin{equation}\label{eq:bkm_balanced_shelling_cond_c}
\overline{C} \ll \overline{B} \text{ and } \overline{A} \cap \overline{B} \subseteq \overline{C} \cap \overline{B} = \overline{B}\setminus \{v\} \text{ with } |\overline{C} \cap V_i|=|b_i| \text{ for } i =1,\dots, m.
\end{equation}
Let~$A$ and~$B$ denote the smallest (w.r.t.~``$\ll$'') facets of~$\Sigma$ containing~$\overline{A}$ and~$\overline{B}$, respectively. So, either $A =B$ or $A \ll B$. If $A=B$, then we are done, since~$\overline{A}$ and~$\overline{B}$ are codimension-one faces of the same facet, implying that they have all but one vertex in common. In this case $\overline{C}= \overline{A}$ satisfies \Cref{eq:bkm_balanced_shelling_cond_c}. Otherwise, if $A\ll B$, then by shellability of~$\Sigma$ there is a (balanced) facet~$C \ll B$ of~$\Sigma$ that satisfies \Cref{eq:bkm_shelling_condition}. In particular $C$ and~$B$ have all but one vertex in common. Both $C$ and $B$ have type~$a$. This implies that the two vertices in $B\triangle C$ must have the same color. Hence if $\{v\} = B \setminus C$ and $\{w\} = C \setminus B$, then there is an $i_0 \in \{1,\dots,m\}$ such that $v,w \in V_{i_0}$. Assume that~$v$ is not a vertex of~$\overline{B}$. Then $\overline{B} \subset C$, leading to a contradiction to the minimality of~$B$. Hence $v$~is a vertex of~$\overline{B}$. Now define $\overline{C} = B \cap C$. Then $\overline{C}$ is a face of $\Sigma$ and  $\overline{C}=\overline{B} \setminus \{v\} \cup \{w\}$. Since $v$ and $w$ are of the same type,  $|\overline{C} \cap V_i|=|b_i|$ holds for all  $i \in [m]$. The fact that $\overline{C}$ is contained in the facet $C\ll B$ of~$\Sigma$ and $\overline{B}\subset B$ is not contained in~$C$ implies that $\overline{C}\ll \overline{B}$. Hence $\overline{C}$ satisfies~\Cref{eq:bkm_balanced_shelling_cond_c}.
\end{proof}

\subsection{Shellability of the two-fold deleted join $(M_r)^{*2}_\Delta$}
\label{subsec:bkm_shelling_of_M_r_2}

We make the following notational conventions. Let~$k$ and~$r$ be integers with~$k \geq 2$ and $r \geq 2k-1$. We write a face $A$ of~$(M_r)^{*k}_\Delta$ as $A=(A_1, \dots, A_k)$, where~$A_i$ lists the vertices used in the $i$-th row of~$(M_r)^{*k}_\Delta$ for $i =1, \dots, k$. We write $A_i = (\overline{A_i}, A_i^r)$, where $\overline{A_i}$ lists the vertices of~$A_i$ contained in the first $r-1$ blocks of~$(M_r)^{*k}_\Delta$  and~$A_i^r$~lists the vertices of~$A_i$ contained in the $r$-th block of~$(M_r)^{*k}_\Delta$. If we need to clarify that a vertex~$v$ of~$\Sigma$ originates from row~$i$, we write~$(v,i)$. We say that a vertex~$v$ of~$(M_r)^{*k}_\Delta$ is \emph{free for a face}~$A$ if there is no vertex of~$A$ in the column containing~$v$, meaning that $A$ contains neither~$v$ nor a copy of~$v$.

Let~$[r]$ refer to a zero-dimensional complex with~$r$ vertices. Then the deleted join $\Delta_{k,r}:=[r]^{*k}_\Delta$ is a ``chessboard complex'' with with~$k$ rows and~$r$ columns; see~\cite{ziegler_shelling_chessboards1994} for a detailed description. Each of the first $r-1$~blocks of~$(M_r)^{*k}_\Delta$ is isomorphic to~$\Delta_{k,r}$. The $r$-fold join~$\Delta_{k,r}^{*r}$ is a subcomplex of $(M_r)^{*k}_\Delta$. The restriction of $\Delta_{k,r}^{*r}$ to the vertices of the first $r-1$~blocks is isomorphic to the $(r-1)$-fold join~$\Delta_{k,r}^{*(r-1)}$. Denote~$\Delta_{k,r}^{*(r-1)}$ by~$\Sigma_{k,r-1}$. Color the vertices of~$\Sigma_{k,r-1}$ based on the row~$i$ they are in.  Let $a=(a_1, \dots, a_k)$ with $a_i = r-1$ for~$i=1,\dots,k$. Then~$\Sigma_{k,r-1}$ is balanced of type~$a$. For $b=(b_1, \dots, b_k)$ with $0 \le b_i \le a_i$, the balanced $b$-skeleton~$\Sigma_{k,r-1}^b$ of~$\Sigma_{k,r-1}$ is the complex given by faces with at most~$b_i$ vertices in row~$i$; see \Cref{def:bkm_balanced_k_skeleton}.

In the following, let $k \geq 2$ and let $r \geq 2k-1$. Then the complex~$\Delta_{k,r}^{*r}$ is shellable. This follows from the fact that the chessboard complex~$\Delta_{k,r}$ is shellable for $r\geq 2k-1$ \cite[Thm.\,2.3]{ziegler_shelling_chessboards1994} and that joins of shellable complexes are shellable (as one can shell the factors ``lexicographically''  \cite[Prop.\,2.4]{provan_billera_vertex_decom1980}).  Since~$\Sigma_{k,r-1}=\Delta_{k,r}^{*(r-1)}$ is a link of the complex~$\Delta_{k,r}^{*r}$. Hence the shelling of~$\Delta_{k,r}^{*r}$ induces a shelling of~$\Sigma_{k,r-1}$.

\begin{remark}\label{rem:bkm_case_M_r_r=2}
For $r=2$, the complex~$(M_2)^{*2}_\Delta$ is not~$(2r-2)$-connected, since its Euler characteristic is~$2$. Hence~$(M_2)^{*2}_\Delta$ is not shellable. A calculation shows that its fundamental group is trivial. 
\end{remark}

\begin{proposition}[Shellability of $(M_r)^{*2}_\Delta$]\label{prop:bkm_shellability_M_r_2}
Let $r \geq 3$ be an integer. Then the $2$-fold deleted join~$(M_r)^{*2}_\Delta$ of the matroid~$M_r$ is shellable.
\end{proposition}
\begin{proof}
For $x=(x_1,x_2)\in \N_{\geq 0}^2$ let $s(x)=(x_{i_1}, x_{i_2})$ be a reordering of the entries of~$x$ by decreasing value such that $x_{i_1} \geq x_{i_2}$. Let $<_l$ be the lexicographic order. Let $\prec$ be the strict order on $\N^2_{\geq 0}$ such that $x \prec y$ if and only if $s(x)<_l s(y)$ or both $s(x)= s(y)$ and $x <_l y$. Let~``$\ll$'' be a shelling of the subcomplex~$\Delta_{2,r}^{*r}$. For facets~$A$ and~$B$~of~$(M_r)^{*2}_\Delta$, let $b,x,y \in \N^2_{\geq 0}$ be defined by
\[
x_i=|A_i^r|, \quad  y_i=|B_i^r|, \quad\text{and} \quad b_i = \min\{r - x_i, r-1\} \quad\text{for }   i=1,2.
\]
If $B \notin\Delta_{2,r}^{*r}$, then let $A \ll B$, if any of the following three cases holds:
\begin{compactenum}[\normalfont (a)]
\item $x \prec y$,
\item $x = y$ and $A^r <_l B^r$,
\item $A^r = B^r$ and $\overline{A} \ll \overline{B}$ for a fixed shelling of the balanced complex~$\Sigma_{2,r-1}^b$.
\end{compactenum}
In the following we show that~``$\ll$'' is indeed a shelling of~$(M_r)^{*2}_\Delta$.Let $A \ll B$ be two facets of~$(M_r)^{*2}_\Delta$. The goal is to find a facet~$C$ that satisfies~\Cref{eq:bkm_shelling_condition}.  We proceed case by case. For clarity, if $A \ll B$ due to~(a), we write $A\ll_{\mathrm{(a)}} B$, likewise for (b) and~(c). 
\\
Case (a): $x \prec y$. Then there is a row~$j \in \{1,2\}$ with $ y_j>1$, implying that $|B_j|>1$. Hence there is a vertex $(b,j) \in B^r \setminus A^r$ and empty block~$t<r$ of~$B$ in row~$j$. We obtain~$C$ by switching~$(b,j)$ with a vertex in the empty block: Let $(c,j)$ be any free vertex for~$B$ in row~$j$ and  block~$t$. Define $C= B \setminus \{(b,j)\} \cup \{(c,j)\}$. Observe that $(|C_1^r|,|C_2^r|) \prec y$. Thus~$C\ll_{\mathrm{(a)}} B$ and~$C$ satisfies~\Cref{eq:bkm_shelling_condition}. 
\\
Case (b): $x=y$ and $A^r <_l B^r$.
Assume $x_2=1$. Then~$x_1>1$. Assume $A_1^r = B_1^r$.  Then~$A^r$ and~$B^r$ each have one vertex in row~$2$ and these two vertices are distinct. Define $C=(\overline{B},A^r)$. Since $A^r$ and $B^r$ only have two rows,~$C$ differs from $B$ in only one vertex. (This is the point where this proof would fail if~$k>2$.) Hence $C\ll_{\mathrm{(b)}} B$ and~$C$ satisfies~\ref{eq:bkm_shelling_condition}.   Assume $A_1^r \neq B_1^r$. Since $x_1>1$, there is an empty block~$t<r$ of~$B$ in row~$1$. We switch vertices: Let~$(b,1)$ be any vertex in~$B_1^r \setminus A_1^r$ and let~$(c,1)$ be any free vertex for~$B$ in block~$t$ and row~$1$. Define $C= B \setminus \{(b,1) \} \cup \{(c,1)\}$ and observe that $|C_1^r| = |B_1^r| -1$ and $|C_2^r| = |B_2^r|$. Hence~$C\ll_{\mathrm{(a)}} B$ and~$C$ satisfies~\ref{eq:bkm_shelling_condition}. Assume $x_2>1$, then there is an empty block~$t<r$ of~$B$ in row~$2$. Again we switch vertices: Let~$(b,2)$ be any vertex in~$B_2^r \setminus A_2^r$ and let~$(c,2)$ be any free vertex for~$B$ in block~$t$ and row~$2$. Define $C= B \setminus \{(b,2) \} \cup \{(c,2)\}$ and observe that $|C_2^r| = |B_2^r| -1$ and $|C_1^r| = |B_1^r|$. Hence~$C\ll_{\mathrm{(a)}} B$ and~$C$ satisfies~\Cref{eq:bkm_shelling_condition}. 
\\
Case (c): Since the balanced subcomplex~$\Sigma_{2,r-1}^b$ is shellable by \Cref{prop:bkm_balanced_subcomplexes_shellable}, there exists a facet $\overline{C}$ of~$\Sigma_{2,r-1}^b$ with $\overline{C}\ll \overline{B}$ that satisfies~\Cref{eq:bkm_balanced_shelling_cond_c}. Define $C=(\overline{C},B^r)$. Then~$C\ll_{\mathrm{(c)}} B$ and~$C$ satisfies~\Cref{eq:bkm_shelling_condition}. 
\end{proof}

For $r\geq 2$, define \emph{blocks}~$E_i'$ by
\[
E_i' = \{v_i^1, \dots, v_i^{r},v_i^{r+1}\} \quad \text{for} \quad i =1,\dots, r-1, \quad \text{and} \quad E_r' = \{w_1, \dots, w_r, w_{r+1}\},
\]
for pairwise distinct~$v_i^j$ and~$w_i$.
Define a matroid~$M'$ with ground set~$\bigcup_i E_i$ by
\[
M' = U_{1,r+1} (E_1) \oplus  \dots \oplus U_{1,r+1} (E_{r-1}) \oplus U_{r+1,r+1} (E_r).
\]
Now let~$M'_{r+1}$ be the $(r-1)$-skeleton of~$M'$. Then~$M'_r$ is ``built by the same principle'' as~$M_r$, but has~$r+1$ instead of only~$r$ disjoint bases. Note that~$(M_r')_\Delta^{*2}$ is a pure complex for all~$r \geq 2$ and contains the $r$-fold join~$\Delta_{2,r+1}^{*r}$ as a subcomplex.  From \cite{ziegler_shelling_chessboards1994} we have that~$\Delta_{2,r+1}^{*r}$ is shellable for all~$r\geq 2$. By starting with a shelling of~$\Delta_{2,r+1}^{*r}$ and repeating the proof of \Cref{prop:bkm_shellability_M_r_2} for~$(M_r')_\Delta^{*2}$ instead of~$(M_r)_\Delta^{*2}$ one obtains the following~corollary.

\begin{corollary}\label{cor:bkm_increase_num_bases_M_r}
For any integer~$r\geq 2$, the $2$-fold deleted join~$(M_r')_\Delta^{*2}$ of the matroid~$M_r'$ of rank~$r$ with~$r+1$ disjoint bases is shellable and hence $(2r-2)$-connected.
\end{corollary}

\subsection{A covering of~$(M_r)_\Delta^{*2}$}
\label{subsec:bkm_covering_of_M_r_2_by_Sigma}

Next we give a topological description of~$(M_r)_\Delta^{*2}$ via a covering by two subcomplexes. This yields a first proof of \Cref{thm:bkm_main_counterexample}. In addition, the covering will allow us to determine the action of the group~$\Z/2:=\langle t \rangle$ on cohomology needed for the proof of~ \Cref{thm:bkm_radon_for_M_r}. 
Recall that the action of $\Z/2$ on $(M_r)_\Delta^{*2}$ is given by interchanging the factors of the join.

\begin{definition}\label{def:bkm_covering_of_M_r_2}
Let~$r \geq 2$ be an integer. Define~$\Sigma_{2r-1}$ to be the subcomplex of~$(M_r)^{*2}_\Delta$ induced by the facets of dimension~$2r-1$ and their faces, and denote by $\Sigma_{2r-2}$ the subcomplex of~$(M_r)^{*2}_\Delta$ induced by the facets of dimension~$2r-2$ and their faces.
\end{definition}

Since each face of~$(M_r)^{*2}_\Delta$ is contained in a facet of dimension~$2r-1$ or~$2r-2$, the two complexes~$\Sigma_{2r-1}$ and~$\Sigma_{2r-2}$ form a covering. The complex~$\Sigma_{2r-1}$ consists of faces that do not have more than~$r-1$ vertices in either row of the last block. The complex~$\Sigma_{2r-2}$ consists of faces that in one row  use only vertices in the last block and in the other row use no vertices in the last block. Hence~$\Sigma_{2r-2}$ has two connected components that we  refer to as~$\Sigma_{2r-2}^1$ and~$\Sigma_{2r-2}^2$; see \Cref{fig:bkm_sigma_2r-2}.

\begin{figure}[h!]
\begin{center}
\includegraphics[scale=.85]{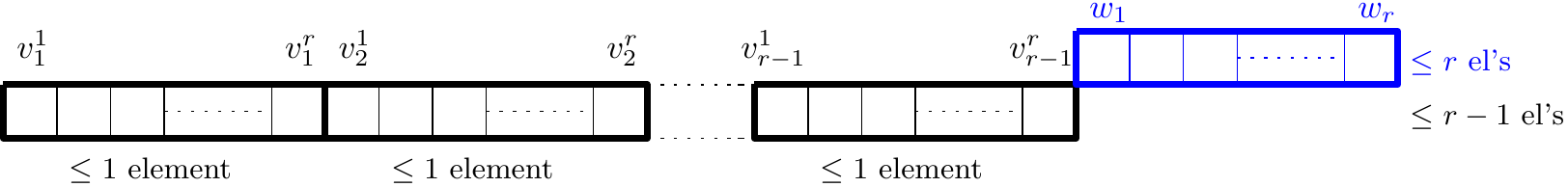}
\caption{One connected component of~$\Sigma_{2r-2}$.}
\label{fig:bkm_sigma_2r-2}
\end{center}
\end{figure}

\begin{proposition} \label{prop:bkm_hpty_type_covering_M_r_2}
Let~$r \geq 3$ be an integer.
\begin{compactenum}[\normalfont (i)]
\item The complex~$\Sigma_{2r-1}$ is a non-trivial  wedge of $(2r-1)$-spheres, and a $\Z/2$-invariant subcomplex. 
\item The complex~$\Sigma_{2r-2}$ is the disjoint union of two contractible spaces $\Sigma_{2r-2}^1$ and $\Sigma_{2r-2}^2$. Moreover, $t\cdot \Sigma_{2r-2}^1=\Sigma_{2r-2}^2$, where $t$~denotes the generator of the group~$\Z/2$.
\item The intersection~$\Sigma_{2r-1} \cap \Sigma_{2r-2}=(\Sigma_{2r-1} \cap\Sigma_{2r-2}^1)\cup (\Sigma_{2r-1} \cap\Sigma_{2r-2}^2)$ is the disjoint union of two non-trivial wedges of $(2r-3)$-spheres, and $t\cdot (\Sigma_{2r-1} \cap\Sigma_{2r-2}^1)=(\Sigma_{2r-1} \cap\Sigma_{2r-2}^2)$.
\end{compactenum}
\end{proposition}
\begin{proof} (i) By~\cite[Lem.\,2.6]{bjorner_wachs_I_nonpure_shellable1996} we can reorder the facets of~$(M_r)^{*2}_\Delta$ by decreasing dimension and obtain a shelling. This is done by first taking the facets of dimension~$2r-1$ in the order given by~``$\ll$,'' then taking the facets of dimension~$2r-2$ again in the order given by~``$\ll$,'' and so forth. This implies that~$\Sigma_{2r-1}$ is shellable. Since the chessboard complex~$\Delta_{2,r}^{*r}$ consists of $(2r-1)$-facets that by definition of~``$\ll$'' are shelled first, homology facets of the chessboard complex are homology facets of~$\Sigma_{2r-1}$. The chessboard complex~$\Delta_{2,r}^{*r}$ is not contractible, and hence~$\Sigma_{2r-1}$ must be a non-trivial wedge of $(2r-1)$-spheres. Since the action of the group~$\Z/2$ is simplicial and therefore preserves the dimension of the simplices we have that $t\cdot \Sigma_{2r-1}=\Sigma_{2r-1}$.
\\
(ii) Each connected component of~$\Sigma_{2r-2}$ is isomorphic to  the join~$(M_r | S)* \Delta_{r-1}$ of the restriction~$(M_r | S)$ and the simplex~$\Delta_{r-1}$, where $S= \bigcup_{i=1}^{r-1} E_i$. Hence each component is contractible. Furthermore, by direct inspection one sees that $t\cdot \Sigma_{2r-2}^1=\Sigma_{2r-2}^2$.
\\
(iii) The intersection~$\Sigma_{2r-1} \cap \Sigma_{2r-2}$ has two connected components. Its faces use in one row only~$r-1$ vertices that are all contained in the last block. In the other row they use no vertices in the last block. Hence both components are isomorphic to the join~$(M_r | S) * \Delta_{r-1}^{(r-2)}$, where~$\Delta_{r-1}^{(r-2)}$ is the $(r-2)$-skeleton of the the simplex. The complex~$(M_r | S)$ is a matroid of rank~$r-1$. It is $(r-2)$-connected and has reduced Euler characteristic~$(r-1)^{r-1}$. Hence each component of~$\Sigma_{2r-1} \cap \Sigma_{2r-2}$ is a non-trivial wedge of $(2r-3)$-spheres.
\end{proof}

\begin{corollary}\label{cor:bkm_htpy_type_M_r_2}
Let $r \geq 3$ be an integer. Then the deleted join~$(M_r)^{*2}_\Delta$ is homotopy equivalent to a non-trivial wedge of spheres of dimensions~$2r-1$ and~$2r-2$.
\end{corollary}

\subsection{Proof of \Cref{thm:bkm_main_counterexample}}
\label{subsec:bkm_proof_main_thm_1}

The homotopy type of a shellable complex can be computed as follows. Let $\Sigma$ be a shellable complex of dimension~$d$. Define the \emph{degree} of a face~$A$ of $\Sigma$ by $\delta(A) = \max\{ |F| \colon  F \in \Sigma, A \subseteq F\}$. Thus $\delta(A)-1$ is the dimension of a largest facet containing~$A$.  Define the \emph{$f$-triangle} $(f_{i,j}(\Sigma))_{0 \le i \le j \le d+1}$ of~$\Sigma$ by  $f_{i,j}(\Sigma) =  |\{ A \in \Sigma \colon |A| = i, \delta(A) =j \}|$. Thus $f_{i,j}(\Sigma)$ is equal to the number of faces~$A$ of~$\Sigma$ of dimension~$i-1$ that are contained in a largest facet of dimension~$j-1$. For $j =0,1, \dots, d+1$ set
\[
h_j(\Sigma) = (-1)^j \cdot \sum_{i=0}^j (-1)^i f_{i,j}(\Sigma) \;.
\]
The vector $h(\Sigma)=(h_0(\Sigma), \dots, h_{d+1}(\Sigma))$ is the diagonal of the ``$h$-triangle'' of~$\Sigma$ \cite[Def.\,3.1]{bjorner_wachs_I_nonpure_shellable1996}. By \cite[Thm.\,4.1]{bjorner_wachs_I_nonpure_shellable1996}, the homotopy type of~$\Sigma$ is a wedge of spheres, consisting of $h_j(\Sigma)$ copies of the $(j-1)$-sphere for $j=1, \dots, d+1$.

\begin{proof}[Proof of \Cref{thm:bkm_main_counterexample}]
The matroid~$M_r$ is of rank~$r$ and has~$r$ disjoint bases. For~$r=2$ the complex~$(M_r)_\Delta^{*2}$ is simply connected, but not $2$-connected; see \Cref{rem:bkm_case_M_r_r=2}. Let~$r\geq 3$ in the following.  Then the complex~$(M_r)_\Delta^{*2}$ is shellable by \Cref{prop:bkm_shellability_M_r_2}. Hence by \cite[Thm.\,4.1]{bjorner_wachs_I_nonpure_shellable1996} the homotopy type of~$(M_r)_\Delta^{*2}$ is a wedge of spheres, consisting of~$h_j$ spheres of dimension $j-1$ for $j=1, \dots, 2r$, where $h((M_r)_\Delta^{*2})=(h_0, \dots, h_{2r})$ is the diagonal of the $h$-triangle of~$(M_r)_\Delta^{*2}$.

For~$j=0, \dots, 2r-2$, the entries~$h_j$ are zero, since~$(M_r)_\Delta^{*2}$ has no facets of dimension~$j$. Hence $(M_r)_\Delta^{*2}$ is $(2r-3)$-connected.

Let $j=2r-1$.  We will show that $h_{2r-1} \neq 0$ and thus that $(M_r)_\Delta^{*2}$ is not $(2r-2)$-connected. The number $h_{2r-1}$ is equal to the alternating sum $-f_{0,2r-1} + f_{1,2r-1} - \dots + f_{2r-1,2r-1}$ of the entries of the row $f_{2r-1}((M_r)_\Delta^{*2}))$ of the $f$-triangle of~$(M_r)_\Delta^{*2}$. Here, $f_{i, 2r-1}$ denotes the cardinality of the set
\[
\mathcal{F}_{i-1}:=\{ A \in (M_r)_\Delta^{*2} \colon |A| = i, \; \delta(A) =2r-1\}
\] 
of faces of dimension $i-1$ of~$(M_r)_\Delta^{*2}$ that are contained in a largest facet of dimension~$2r-2$. Facets~$A$ of dimension~$2r-2$ are given by choosing all vertices $\{w_1, \dots, w_r\}$ in one row of~$(M_r)_\Delta^{*2}$ and one vertex from each of the sets $\{v_i^1, \dots, v_i^r \}$ for $i=1, \dots, r-1$ in the other row. This implies that $f_{i, 2r-1} =0$ for $i=0, \dots, r-1$ and that any face in $\mathcal{F}_{i-1}$ for $i \geq r$ must use all vertices $\{w_1, \dots, w_r\}$. See \Cref{fig:bkm_matroid_deljoin_m5_8_facet} for an example of a facet of dimension $2r-2=8$ for~$r=5$.  

Let $S=\{v_i^1, \dots, v_i^r : 1 \le i \le r-1 \}$ and let $M_r | S$ be the restriction of~$M_r$ to the vertex set~$S$. Then for $i=0, \dots,r-1$ there is a $2$-to-$1$ surjection between the faces in~$\mathcal{F}_{r+i-1}$ and the set of $(i-1)$-dimensional faces of~$M_r | S$. This implies that 
\[
f_{r+i,2r-1}((M_r)_\Delta^{*2})= 2\, f_{i}(M_r|S)\quad  \text{for}\quad  i=0, \dots,r-1,
\]
where~$f_{i}(M_r|S)$ is the number of $(i-1)$-faces of~$M_r|S$. Hence
\[
h_{2r-1} = (-1)^{r-1}\,2\,\big(\chi(M_r|S)-1\big),
\]
where $\chi(M_r|S)$ denotes the Euler characteristic of~$M_r|S$.

The complex~$M_r|S$ is isomorphic to the $(r-1)$-fold join of the restriction $M_r|\{v_1^1, \dots, v_1^r \}$, which in turn is isomorphic to a $0$-dimensional complex with~$r$ vertices. Hence $M_r|S$ has Euler characteristic equal to $1 + (-1)^{r-1} (r-1)^{r-1}$. Thus $h_{2r-1} =2\,(r-1)^{r-1}$, which is non-zero since~$r\geq 2$.
\end{proof}

The missing value~$h_{2r}$ of the diagonal of the $h$-triangle can be calculated, similarly to the above, by calculating the complete $f$-vector of~$(M_r)_\Delta^{*2}$. Both calculations are technical. Instead we give lower bounds. 

\begin{corollary}\label{cor:bkm_betti_nums_M_r_2}
Let $r\geq 3$ be an integer and let~$\beta_i$ denote the $i$-th reduced Betti number of~$(M_r)_\Delta^{*2}$ for $i=0,\dots, 2r-1$. Then
\[
\beta_i = \begin{cases}
2\,(r-1)^{r-1} &\text{ if } i = 2r-2\\
0 &\text{ if } i \le 2r-3
\end{cases} \qquad \text{and} \qquad  
\beta_{2r-1} \geq (r^2 -3 r +1)^r\,.
\]
\end{corollary}
\begin{proof}
Note that $\beta_i = h_{i+1}$ and that $h_{2r}$ is equal to the number of $(2r-1)$-dimensional homology facets of~$(M_r)_\Delta^2$.
In the shelling of~$(M_r)_\Delta^2$ the $r$-fold join~$\Delta_{2,r}^{*r}$ of the chessboard complex is shelled first, implying that its homology facets are also homology facets of~$(M_r)_\Delta^2$. The chessboard complex is pure and has Euler characteristic~$1 - (r^2-3r+1)^r$. Therefore $h_{2r} \geq (r^2 -3 r +1)^r$.
\end{proof}

\section{Further results}
\label{sec:bkm_further_results}

\subsection{Bounds for the topological Tverberg number of matroids}
\label{subsec:bkm_bounds_top_tverbeg_num}

Recall that the \emph{topological Tverberg number $\TT(M,d)$} of~$M$ is the largest integer~$k\geq1$ such that for every continuous map $f \colon M \rightarrow \R^d$, there is a collection $\{\sigma_1, \dots, \sigma_k\}$ of~$k$ pairwise disjoint faces, called a Tverberg $k$-partition, such that $\bigcap_{i=1}^k f(\sigma_i) \neq \emptyset$.

\begin{corollary}[Lower bounds for the topological Tverberg number]\label{cor:lower_bounds_TT_num}
Let~$b,d,r \geq 1$ be integers and let~$M$ be a matroid of rank~$r$ with~$b$ disjoint bases. Let~$x=d+1$ for ease of notation.
Let
\[
\ell(b,r,x) =\frac{ 2x + (r-x)b+ \sqrt{(2x + b(r-x))^2 + 8bx^2} }{8x}.
\]
If $p$ is a prime power with 
\[
p \le 2\ell(b,r,x),
\] 
then $\TT(M,d)\geq p$.
\end{corollary}
\begin{proof}
We use the join scheme and take a connectivity-based approach based on the lemma in~\cite{volovikov_on_a_top_general1996} due to Volovikov, which can be seen as a generalization of Dold's theorem. If we show that the connectivity of the configuration space~$M^{*p}_\Delta$ is at least as high as the dimension of the test space~$S^{(p-1)(d+1)-1}$, then the result follows.

By \cite[Cor.\,3]{barany_kalai_meshulam2016} the deleted join~$M^{*p}_\Delta$ has connectivity at least $br/(\lceil b/p\rceil +1) -2$, implying that its connectivity is at least $\lceil br/(\lceil b/p\rceil +1)\rceil -2$. Hence it suffices to show that 
\[
\frac{br}{b/p +2} - (p-1)(d+1) \geq 0. 
\]
This is equivalent to 
\begin{equation}\label{eq:bkm_prime_p_lower_bound_tverb_num}
-2xp^2 + (2x-xb +br)p + xb \geq 0,
\end{equation}
which defines a negatively curved parabola in~$p$ with zeros
\[
\frac{2x + (r-x)b + br \pm \sqrt{(2x + (r-x)b)^2 + 8bx^2}}{4x}\;.
\]
Finally, we observe that 
\[
\frac{2x + (r-x)b - \sqrt{(2x + (r-x)b)^2 + 8bx^2}}{4x} \; \le \;  \ell(b,r,x) \;  \le \;  p \;\le 2\ell(b,r,x),
\]
and hence~$p$ satisfies \Cref{eq:bkm_prime_p_lower_bound_tverb_num}.
\end{proof}

Upper bounds for the topological Tverberg number for matroids with codimension at least~$3$, meaning $r-1 \le d-3$, can be obtained using the new sufficiency criterion \cite[Thm.\,7]{mabillard_wagner_elim_I_2015} due to Mabillard and Wagner  for the non-existence of Tverberg $k$-partitions for simplicial complexes with codimension~$3$. For a real number $x \geq 0$, we let~$\lceil x\rceil_\mathrm{npp}$ denote the smallest integer~$k\geq x$ that is not a prime-power.

\begin{proposition}[Upper bounds for the topological Tverberg number] \label{prop:bkm_upper_bounds_TT_num}
Let $d\geq 3$ and $r\geq 1$ be integers and let $r \le d-2$. If~$M$ is a matroid of rank~$r$, then
\[
\left\lceil \tfrac{d}{d-r+1} \right \rceil_\mathrm{npp} > \mathrm{TT}(M,d).
\]
\end{proposition}
\begin{proof}
Let $k= \left\lceil \tfrac{d}{d-r+1} \right \rceil_\mathrm{npp}$. The dimension of the deleted product $M^{\times k}_\Delta$ is at most~$k(r-1)$, which by choice of~$k$ is at most~$(k-1)d$. Thus by \cite[Cor.\,5.2]{blagojevic_ziegler_beyond_borsuk_ulam2017}, which is a simple consequence of \cite[Lem.\,4.2]{ozaydin_1987}, there exists an $\mathfrak{S}_k$-equivariant map from  $M^{\times k}_\Delta$  to the sphere~$S^{(k-1)d-1}$. (Here~$\mathfrak{S}_k$ denotes the symmetric group on~$k$ letters.) Since~$M$ has dimension at most~$d-3$, we can apply \cite[Thm.\,7]{mabillard_wagner_elim_I_2015} and get the existence of a continuous map $f\colon M \rightarrow \R^d$ that does not have a Tverberg $k$-partition. This implies that~$k > TT(M,d)$.
\end{proof}

\begin{remark}\label{rem:bkm_patak_tverberg_type}
Recently Paták \cite{patak_tverberg_matroids2017} building on \cite{goaoc_mabillard_patak_wagner_etc_heawood_ineq2016} proved several Tverberg-type results for matroids that are not directly related to \cite[Thm.\,1]{barany_kalai_meshulam2016}, including colored versions. We point out~\cite[Lem.\,2]
{patak_tverberg_matroids2017}: Let~$M$ be a matroid of rank~$r \geq 1$ with closure operator $\mathrm{cl}$ and let $S$ be a subset of the ground set of~$M$ of cardinality at least $r(k-1)+1$. Then there exist pairwise disjoint subsets $S_1, \dots, S_k$ of $S$ such that $\mathrm{cl} \,\emptyset \subsetneq \mathrm{cl} S_1 \subseteq  \dots \subseteq \mathrm{cl} S_k$. 
\end{remark}

\subsection{Connectivity of the deleted product of a matroid}
\label{subsec:bkm_conn_del_prod}

In \Cref{thm:bkm_conn_del_prod} we assume that the rank~$r$ of~$M$ is at least~$k$, otherwise~$M^{\times k}_\Delta$ can be empty. To simplify the statement of the theorem we assume that $k \le b$, since then the dimension of~$M^{\times k}_\Delta$ is equal to~$(r-1)k$ and, in particular, is independent of the number of disjoint independent sets of lower cardinality. We point out, however, that the proof of \Cref{thm:bkm_conn_del_prod} can  be applied to the setting where $k<b$; see for example~\Cref{cor:bkm_conn_simplex}. 

\begin{theorem}[Connectivity bounds for the deleted product]\label{thm:bkm_conn_del_prod}
Let $b,k,r \geq 2$ be integers with $r \geq k$ and~$b\geq k$. Let~$M$ be a matroid of rank~$r$ with~$b$ disjoint bases and let~$M^{\times k}_\Delta$ be the $k$-fold deleted product of~$M$.
\begin{compactenum}[\normalfont (i)]
\item Then the connectivity of~$M^{\times k}_\Delta$ is at least 
\[
 r - 2 - \left\lfloor \frac{r(k-1)}{b} \right \rfloor \,.
\]
\item If $b\geq r(k-1) +1$, then~$M^{\times k}_\Delta$ is $(r-2)$-connected, but not $(r-1)$-connected.
\end{compactenum}
\end{theorem}

The \emph{ordered configuration space $\mathrm{Conf}(X,n)$} of~$n$ particles in a topological space~$X$ is defined as the space $ \{(x_1, \dots, x_n) \in X^n \colon x_i \neq x_j \text{ for } i \neq j \}$. As Smale \cite[Lem.\,2.1]{shapiro_obstructions_I_1957} observed, in the case where $n=2$ and $X=\Sigma$ is a finite simplicial complex, the $2$-fold deleted product $\Sigma^{\times 2}_\Delta$ is a deformation retract of~$\mathrm{Conf}(M,2)$. This leads to the following corollary of \Cref{thm:bkm_conn_del_prod}.

\begin{corollary}\label{cor:bkm_conn_config_space}
Let $b,r\geq2$ be integers and let~$M$ be a matroid of rank~$r$ with~$b$ disjoint bases. Then the  configuration space $\mathrm{Conf}(M,2) = \{(x,y) \in M^2 \colon x \neq y \}$ of two ordered particles in~$M$ is at least
\[
\Big(r -2 - \left\lfloor \frac{r}{b} \right \rfloor \Big)\text{-}
\]
connected and not $(r-1)$-connected, when $b\geq r+1$.
\end{corollary}

By \cite[Thm.\,3.2.1]{provan_billera_vertex_decom1980} any matroid~$M$ of rank~$r$ is pure shellable, implying that~$M$  is contractible or homotopy equivalent to a wedge of $(r-1)$-spheres. The reduced Euler characteristic~$\widetilde{\chi}(M)$ of~$M$ can be computed using the M\"obius function~$\mu$ of the lattice of flats~$L$ of the dual matroid of~$M$. By \cite[Prop.\,7.4.7]{bjorner_matroid_applications_1992} $\widetilde{\chi}(M)$ is zero if $M$ has coloops (elements contained in every basis), and is otherwise equal to $(-1)^{r-1}|\mu_L(\hat{0},\hat{1})|$, which is non-zero. In fact $|\mu_L(\hat{0},\hat{1})|$ is non-zero for any geometric lattice \cite[Thm.\,4]{rota_foundations_comb_theory_1_1964}. 

\begin{proof}[Proof of \Cref{thm:bkm_conn_del_prod}]
(i) A cell of~$M^{\times k}_\Delta$ is of the form
\[
\mathrm{relint}(\sigma_1)\times \dots \times \mathrm{relint}(\sigma_k),
\]
where $\sigma_i \cap \sigma_j = \emptyset$ for all $i, j$ with $1\le i<j \le k$. Its dimension is given by the sum of the dimensions of the~$\sigma_i$. Since the~$\sigma_i$ are vertex-disjoint and by assumption~$k \le b$, a product cell of maximal dimension uses~$rk$ vertices and has dimension~$(r-1) k$. 

We fix~$r$ and establish the connectivity of~$M^{\times k}_\Delta$ by induction on~$k$. Assume $k=1$. If~$M$ has no coloops, it is homotopy equivalent to a wedge of $(r-1)$-spheres, else it is contractible.

Assume the statements of the theorem are true for $k-1$ for a fixed $k\geq 2$. Consider the projection~$p_{k-1}$ of the $k$-fold product~$M^k$ to the first $k-1$~coordinates. The map $p_{k-1}$ restricts to a surjective continuous proper map 
\[
p_{k-1} \colon M^{\times k}_\Delta \longrightarrow M^{\times k-1}_\Delta.
\]
Since~$r\geq k$, both the domain and codomain of~$p_k$ are connected by induction.  They are also locally compact, locally contractible separable metric spaces.

Let $x \in M^{\times k-1}_\Delta$ be a point and let faces $\sigma_1, \dots, \sigma_{k-1} \in M$ be minimal under inclusion such that~$x$ is contained in the product $ \mathrm{relint}(\sigma_1)\times \dots \times \mathrm{relint}(\sigma_{k-1})$ of the relative interiors of the~$\sigma_i$. Let $ V_x= \mathrm{vert}(\sigma_1) \cup  \dots \cup \mathrm{vert}(\sigma_{k-1})$ be the union of the vertex sets of the~$\sigma_i$.  Assume $V_x =\{v_1, \dots, v_n\}$. Then the preimage
\begin{align*}
p_{k-1}^{-1} (\{x\})&\cong \{y \in M \colon \exists\, \sigma \in M  \text{ s.t. } y \in \sigma \text{ and } \sigma\cap \sigma_i = \emptyset \text{ for } i=1,\dots,k-1 \}\\
&=\{ \sigma \in M \colon \{v_i \}  \not\subseteq \mathrm{vert}(\sigma) \text{ for } i =1,\dots, n\}.
\end{align*}
Hence~$p_{k-1}^{-1} (\{x\})$ is homeomorphic to (any geometric realization of) the successive deletion $M_x :=M\setminus v_1 \setminus  \dots \setminus v_n$ of the vertices~$V_x$ from~$M$. Any deletion of a matroid is again a matroid (see \cite{oxley_matroid_theory_2011}), implying by induction that~$M_x$ is a matroid. Let~$r_x$ denote its rank. The total number~$n$ of vertices deleted is at most~$r(k-1)$. Let
\[
r_k =r - \left\lfloor \frac{r(k-1)}{b} \right\rfloor.
\] 
If~$V_x$ contains $\lfloor r(k-1)/b \rfloor$ vertices from $b$~disjoint bases of~$M$, then~$r_x$ can be equal to~$r_k$, otherwise~$r_x$ is larger. Equality is given, if $M$ has no other disjoint independent sets of cardinality greater than~$r_k$. Thus $M_x$ is a matroid of  rank~$r_x \geq r_k$. Hence~$p_{k-1}^{-1} (\{x\})$ is locally contractible and either contractible (if~$M_x$ has coloops) or homotopy equivalent to a wedge of $(r_x-1)$-spheres, which is at least $(r_k-2)$-connected. By induction hypothesis $M^{\times k}_\Delta$ is $(r_{k-1} -2)$-connected. Since $r_k \le r_{k-1}$, the deleted product~$M^{\times k}_\Delta$ is~$(r_k-2)$-connected by Smale's theorem, which is stated below.
\begin{quote}
{\small
\textbf{Smale's Theorem \cite{smale_vietoris_mapping1957}}. \emph{Let $X$ and $Y$ be connected, locally compact, separable metric spaces, and in addition let $X$ be locally contractible.
	Let $f:X\longrightarrow Y$ be a surjective continuous proper map.
	If for every $y\in Y$ the preimage $f^{-1}(\{y\})$ is locally contractible and $n$-connected, then the induced homomorphism
	\[
	f_{\#} :  \pi_i(X)\rightarrow \pi_i(Y)
	\]
	is an isomorphism for all $0\leq i\leq n$, and is an epimorphism for $i=n+1$.}
}
\end{quote}

\noindent
(ii) Let $b \geq (r-1)(k-1) +1$ for fixed~$r$ and fixed~$k$. Then $r_k=r$. Hence~$p_{\ell-1}^{-1} (\{x\})$ is a matroid of rank~$r$ for all $x \in  M^{\times \ell-1}_\Delta$ and all $\ell$ with~$2\le \ell \le k$.  For $\ell=2$, the deleted product~$M^{\times \ell-1}$ is equal to the matroid~$M$, which is $(r-2)$-connected and not $(r-1)$-connected, since~$M$ has no coloops. By induction on~$\ell$ the skeleton $M^{\times \ell-1}_\Delta$ is $(r-2)$-connected but not $(r-1)$-connected. Hence by Smale's theorem~$M^{\times \ell}_\Delta$ is $(r-2)$-connected, but not $(r-1)$-connected. 
\end{proof}

As a corollary we obtain a proof of the following result due to B{\'{a}}r{\'{a}}ny, Shlosman, and Sz{\H{u}}cs.

\begin{corollary}[{\cite[Lem.\,1]{barany_shlosman_topol_tverberg1981} }]\label{cor:bkm_conn_simplex}
Let $r$ and~$k$ be integers with $r \geq k \geq 1$. Then the deleted product~$(\Delta_{r-1})^{\times k}_\Delta$ of the simplex~$\Delta_{r-1}$ of dimension~$r-1$ is $(r-k-1)$-connected and not $(r-k)$-connected.
\end{corollary}
\begin{proof}
The simplex $\Delta_{r-1}$ is a uniform matroid~$U_{r,r}$ of rank~$r$. It has one basis. Its dimension~$d_k$ is equal to $r - k$ for all~$k\geq 1$. The fibers $p_{k-1}^{-1} (\{x\})$ are all contractible, since vertex-deletions of the simplex are contractible. Hence for all~$k\geq2$, Smale's theorem  together with the Whitehead theorem implies that $(\Delta_{r-1})^{\times k}_\Delta$ and  $\mathrm{sk}_{r-k}\big((\Delta_{r-1})^{\times k}_\Delta\big)$ are homotopy equivalent. In particular~$(\Delta_{r-1})^{\times k}_\Delta$ is $(r-k-1)$-connected and not $(r-k)$-connected.
\end{proof}

\subsection{A topological Radon-type theorem for $M_r$}
\label{subsec:bkm_topological_radon_thm}

The following topological Radon-type theorem for the family of matroids $M_r$ ($r\geq 3$) follows from \Cref{thm:bkm_main_counterexample} by using the join scheme and taking the connectivity-based approach.

\begin{corollary}\label{cor:bkm_radon_for_M_r_via_connectivity}
Let $d\geq 1$ and $r\geq 3$ be integers such that $2r -3\geq  d$. Then $\TT(M_r,d) \geq 2$.
\end{corollary}
\begin{proof}
By \Cref{thm:bkm_main_counterexample} the connectivity of the configuration space~$(M_r)^{*2}_\Delta$ is~$2r-3$, which is at least as high as the dimension of the test space~$S^d$. 
\end{proof}

We obtain a sharper result by  using the join scheme and applying \cite[Thm.\,1]{blagojevic_blagojevic_mcclary_equi_tian_jordan_curve}, which is obtained by a Fadell--Husseini index calculation. We point out the following typo in \cite[Thm.\,1]{blagojevic_blagojevic_mcclary_equi_tian_jordan_curve}: In the notation of the theorem, the roles of~$X$ and~$Y$ should be interchanged in the last sentence of the statement. See \cite[Thm.\,4.2]{blagojevic_lueck_ziegler_equiv_top_config_spaces2015} for a more general version that implies \cite[Thm.\,1]{blagojevic_blagojevic_mcclary_equi_tian_jordan_curve}.

\begin{proof}[Proof of \Cref{thm:bkm_radon_for_M_r}]
Without loss of generality let $d=2r-2$. In order to prove the theorem using the join scheme, we need to show that there is no $\Z/2$-equivariant map $(M_r)^{*2}_\Delta \to S^{d}$, where the sphere is equipped with the antipodal action. 

From \Cref{cor:bkm_htpy_type_M_r_2} we have that $H^{d}((M_r)^{*2}_\Delta;\F_2)\neq 0$, and $H^{i}((M_r)^{*2}_\Delta;\F_2)= 0$ for all~$i$ with $1\leq i\leq d-1$. Hence $(M_r)^{*2}_\Delta$ is not $d$-connected. Consequently the classical Dold theorem~\cite{dold_simple_proofs_borsuk1983} cannot be applied. To prove non-existence of a $\Z/2$-equivariant map we use~\cite[Thm.\,1]{blagojevic_blagojevic_mcclary_equi_tian_jordan_curve} For this it suffices to prove that the cohomology $H^{d}((M_r)^{*2}_\Delta;\F_2)$ is a free $\F_2[\Z/2]$-module where the action is induced by the $\Z/2$-action on~$(M_r)^{*2}_\Delta$.

Indeed, consider the covering $\{\Sigma_{d+1},\Sigma_{d}\}=\{\Sigma_{d+1},\Sigma_{d}^1\cup\Sigma_{d}^2\}$ of the complex $(M_r)^{*2}_\Delta$; see \Cref{subsec:bkm_covering_of_M_r_2_by_Sigma}. The relevant part of the induced Mayer--Vietoris sequence in cohomology with $\F_2$-coefficients has the form:
\[
\xymatrix@C=1pc @R=.5pc{
H^{d-1}(\Sigma_{d+1})\oplus H^{d-1}(\Sigma_{d}^1)\oplus H^{d-1}(\Sigma_{d}^2)\ar[r]&
H^{d-1}(\Sigma_{d+1}\cap \Sigma_{d}^1)\oplus H^{d-1}(\Sigma_{d+1}\cap \Sigma_{d}^2)\ar[r]&  \\
H^{d}((M_r)^{*2}_\Delta)\ar[r]&
H^{d}(\Sigma_{d+1})\oplus H^{d}(\Sigma_{d}^1)\oplus H^{d}(\Sigma_{d}^2).
}
\]
From Proposition \ref{prop:bkm_hpty_type_covering_M_r_2} we have that the subcomplexes $\Sigma_{d}^1$ and $\Sigma_{d}^2$ are contractible, and that $\Sigma_{d+1}$ is $d$-connected.
Thus the sequence simplifies to:
\[
\xymatrix@C=1pc @R=.5pc{
0\ar[r]&
H^{d-1}(\Sigma_{d+1}\cap \Sigma_{d}^1)\oplus H^{d-1}(\Sigma_{d+1}\cap \Sigma_{d}^2)\ar[r]&  
H^{d}((M_r)^{*2}_\Delta)\ar[r]&
0.
}
\]
Since, again by Proposition \ref{prop:bkm_hpty_type_covering_M_r_2}, the $\Z/2$-action interchanges the subcomplexes $\Sigma_{d+1}\cap \Sigma_{d}^1$ and~$\Sigma_{d+1}\cap \Sigma_{d}^2$, we conclude that $H^{d}((M_r)^{*2}_\Delta;\F_2)$ is a free $\F_2[\Z/2]$-module, as claimed.
\end{proof}

\subsection{Failure of shellability and vertex-decomposability for general~$k$}
\label{subsec:bkm_general_shellability_vertex_decomp}

For a definition of vertex-decomposability for possibly non-pure complexes see \cite[Def.\,11.1]{bjorner_wachs_II_nonpure_shellable1997}.

\begin{proposition}\label{prop:bkm_M_r_k_not_shellable_not_VD}\mbox{} Let~$k$ and~$r$ be integers.
\begin{compactenum}[\normalfont (i)]
\item For $k\geq 3$ and $r\geq 2k-1$ the complex $(M_r)^{*k}_\Delta$ is not shellable.\label{prop:bkm_M_r_k_not_shellable_not_VD_i}
\item For~$k\geq 2$ and $r\geq 2k-1$ the complex $(M_r)^{*k}_\Delta$ is not vertex-decomposable.\label{prop:bkm_M_r_k_not_shellable_not_VD_ii}
\end{compactenum}
\end{proposition}
\newpage
\begin{proof}
(\ref{prop:bkm_M_r_k_not_shellable_not_VD_i}) Let $r=2k-1$. Consider the face
\[
A= \{(v_1^i,i),\dots, (v_{r-1}^i,i) \; \colon \; i=1, \dots,k-1 \} \cup \{(v_1^r,k),\dots, (v_{k-1}^r,k), (w_{k},k), \dots, (w_r,k) \}.
\]
Hence~$A$ has in rows~$1$ to $k-1$ one vertex in each of the first $r-1$~blocks and  in the $k$-th row $k-1$~vertices in the first~$r-1$ blocks and~$r-k+1$ vertices in the last block. The link~$(M_r)^{*k}_\Delta / A$ of~$A$ is isomorphic to the square chessboard complex~$\Delta_{k-1,k-1}$, which is not (pure) shellable by \cite[Thm.\,2]{friedman_hanlon_betti_num_chessboards1998}. However, links of  shellable complexes must be shellable \cite[Prop.\,10.14]{bjorner_wachs_II_nonpure_shellable1997}.
\\
(\ref{prop:bkm_M_r_k_not_shellable_not_VD_ii}) To see that~$(M_r)^{*k}_\Delta$  is not vertex-decomposable, we argue by contradiction. First we point out that it suffices to show this statement for $k=2$, since by \cite[Thm.\,11.3]{bjorner_wachs_II_nonpure_shellable1997}, vertex-decomposability implies shellability. Assume there is a shedding sequence~$S$ for~$(M_r)^{*2}_\Delta$. Consider only the deletions and let~$s_0\in S$ be the first vertex to be deleted from block~$r$ of~$(M_r)^{*2}_\Delta$. Let $M'$ be the complex given by successively deleting all vertices up to but not including~$s_0$. By symmetry we may assume that~$s_0=(w_1,2)$.  Let~$d$ be the dimension of $M'$.  Consider the link~$M' / s_0$. A facet of the link has dimension~$d-1$ and uses one vertex less in the second row than a $d$-dimensional facet of~$M'$. It cannot use the vertices~$(w_1,1)$ or~$(w_1,2)$. Let~$A$ be a facet of~$M' / s_0$ that uses the vertices~$(w_i,1)$ for $i=2,\dots,r$ in  the first row of block~$r$. Then the vertices~$(w_i,2)$ for $i=1,\dots,r$  in the second row of block~$r$ cannot be used by~$A$, since they are either deleted ($i=1$) or ``blocked'' ($i>1$). Now consider the deletion~$M'\setminus s_0$.  It has dimension~$d$. None of its facets use the vertex~$(w_1,2)$ and some of its facets have dimension~$d-1$. In fact~$A$ is a facet of~$M'/s_0$ of dimension~$d-1$. Hence~$A$ is a facet of both the link and the deletion of~$M'$ with respect to~$s_0$, thus violating the definition of vertex-decomposability \cite[Def.\,11.1]{bjorner_wachs_II_nonpure_shellable1997}.
\end{proof}

{\setstretch{1.2}
\addcontentsline{toc}{chapter}{Bibliography}

}
\end{document}